\documentclass{article}
\usepackage{amsmath}
\usepackage{amssymb}
\usepackage{mathrsfs}
\usepackage{amsthm}
\usepackage{paralist}
\usepackage{graphics} 
\usepackage{epsfig} 
\usepackage{graphicx}
\usepackage{epstopdf}
\usepackage{multirow}
\usepackage[colorlinks=true]{hyperref}
\hypersetup{urlcolor=blue, citecolor=red}

\graphicspath{{fig/}}

\numberwithin{equation}{section}
\newtheorem{theorem}{Theorem}
\newtheorem{lemma}{Lemma}
\newtheorem{remark}{Remark}
\numberwithin{theorem}{section}
\numberwithin{lemma}{section}
\numberwithin{remark}{section}
\title{Nitsche-XFEM for optimal control problems governed by elliptic PDEs with interfaces
\thanks
{
	This work was supported by  National Natural Science Foundation of China (11771312).
}}
\author{
	 \ Tao Wang  \thanks{Email: wangtao5233@hotmail.com }, Chaochao Yang \thanks{Email: yangchaochao9055@163.com}, \
	Xiaoping Xie \thanks{Corresponding author. Email: xpxie@scu.edu.cn} \\
	{School of Mathematics, Sichuan University, Chengdu 610064, China}
}

\begin{document}
	\maketitle

	\begin{abstract}
		For the optimal control problem governed by elliptic equations with interfaces, we present a numerical method based on the Hansbo’s Nitsche-XFEM $\cite{Hansbo2002An}$. We followed the Hinze’s variational discretization concept $\cite{Hinze05var}$ to discretize the continuous problem on a uniform mesh. We derive optimal error estimates of the state, co-state and control both in mesh dependent norm and $L^2$ norm. In addition, our method is suitable for the model with non-homogeneous interface condition. Numerical results confirmed our theoretical results, with the implementation details discussed.
	\end{abstract}	
	
	\section{Introduction}
	In material science or engineering design, when optimizing physical processing composed of several materials with different conductivities or densities, we will encounter optimization problems governed by PDEs with interfaces. We consider the following  linear quadratic optimal control problem: 
	\begin{equation}\label{eqobjective}
	\text{min}~ J(y, u):=\frac{1}{2}\int_\Omega (y-y_d)^2~dx+\frac{a}{2}\int_\Omega u^2~ds 
	\end{equation}
	for $ (y,u)\in H^1_0(\Omega)\times L^2 (\Omega) $ subject to the elliptic interface problem 
	\begin{equation}\label{eqstrongstate}
	\left\{
	\begin{array}{rll}
	& -\nabla\cdot(\alpha(x)\nabla y)=u, & \text{ in }\Omega \\
	& y=0, & \text{ on }\partial\Omega \\
	& [y]=0,[\alpha\nabla_n y ]=g, & \text{ on }\Gamma\\
	\end{array}
	\right.
	\end{equation}
	with the control constraint
	\begin{equation}\label{eqconstraint}
	u_a\leq u \leq u_b, \text{ a.e. in } \Omega.
	\end{equation}
	
	Where $\Omega$ is a bounded domain in $\mathbb{R}^2$ with convex polygonal boundary $\partial\Omega$, and an internal smooth interface $\Gamma$ dividing $\Omega$ into two open sets $\Omega_1$ and $\Omega_2$. $y_d\in L^2(\Omega)$ is the desired state to be achieved by controlling  $u$, and $a$ is a positive constant. $\alpha$ is a piecewise constant. $\alpha|_{\Omega_i}=\alpha_i>0$ for $i=1,2$.
	$[y]=y|_{\Omega_1}-y|_{\Omega_2}$ denotes the jump of the function $y$, $\textbf{n}$ is normal vector of $\Gamma$ pointing to $\Omega_1$. $\nabla_\textbf{n} y$ is normal derivative of $y$ on $\Gamma$. $g\in H^{\frac{1}{2}} (\Gamma)$, and  $u_a, u_b \in L^2 (\Omega)$ with $u_a\leq u_b$ a.e. in $\Omega$. For the sake of simplicity, we choose homogeneous
	boundary condition on $\partial\Omega$, since similar results can be obtained for other boundary
	conditions.
	
	For the elliptic interface problem, it is well-known that the solution $y$ of problem \eqref{eqstrongstate} generally not in $H^2(\Omega)$. It may leads to the reduced accuracy for numerical approximations $\cite{Babu1970The,Xu1982Estimate}$.  \cite{ Barretts87Fit, Brambles96fin, Chen98, Huangs02mor, Plums03opt, Lis10opt, Xus16opt, Cai17dis} used body(interface)-fitted mesh to improve the accuracy. However, it is often difficult or expensive to generate complicated body-fitted mesh, and we have to update the mesh, when the interface is moving with time or iteration. Unfitted finite element methods are designed to conquer these difficulties. In general, they use some special shape functions to improve the approximation property of the shape function space. IFEM (Immersed Finite Element Method) $\cite{Li03new, Camps06qua, Li06imm, Fedkiw06imm, Lins07err, Lins15par}$ used some special shape functions, which satisfy interface conditions exactly. It is easy to see the original IFEM is not suitable for the model with non-homogeneous interface condition $[\alpha\nabla_n y ]=g$, because it requires basis functions satisfy the interface conditions exactly. The special IFEM $\cite{He2011Immersed,Han2016A}$ introduce some additional special basis functions to handle the non-homogeneous interface condition, but they don't give the numerical analysis. In $\cite{Yan2007Immersed}$, they constructed  a single function satisfies the non-homogeneous jump conditions by using a level-set representation of the interface, and solve the elliptic interface problem with homogeneous interface condition. The another kind of unfitted finite element method is the XFEM (eXtended Finite Element Method) or another name GFEM (Generalized Finite Element Method). XFEM use additional special basis functions to mimic the local behavior of the exact solution, it has developed for many problems. For the development of the XFEM, we refer to \cite{Babu1994Special,Babu2015THE,Strouboulis2000The,Strouboulis2006The,I2012Stable,moes1999a,Belytschko2009TOPICAL,Nicaise2011Optimal}. In \cite{Kergrene2016Stable,Soghrati2012An,Soghrati2012A,Cheng2010Higher}, XFEM is applied to the elliptic interface problem, but they are only focus on the numerical simulation, they do not give the numerical analysis in their paper.
	
	It should be mentioned that in  \cite{Hansbo2002An}, a special unfitted finite element method was proposed for the elliptic interface problem. This method used additional cut basis functions and coupled with a variant of Nitsche's approach  \cite{Nitsche71ube}. Cut basis functions are some piecewise linear functions discontinuous across the interface, these basis functions are cut at the interface, which improve the approximation property of the shape function space. And the formulation of this method is symmetric positive definite and consistent (in some sense) by using Nitsche's approach. Optimal error analysis are given for the elliptic interface problem with non-homogeneous interface condition in \cite{Hansbo2002An}, which is second order for $L^2$ norm. There are many names of this method, it named as CutFEM in \cite{Hansbo14cut, Burman15cut, Cenanovics16cut, Schott17sta}, as Nitsche-XFEM in \cite{Lehrenfeld2014Optimal,Barrau2012A,Becker2009A}, as DG-XFEM in $\cite{Lehrenfeld2012Analysis,Wang2016High,Lehrenfeld2014The}$.	We call it as Nitsche-XFEM.
	
\begin{figure}[htpb]
	\centering
	\begin{minipage}{5cm}
		\includegraphics[width=5cm]{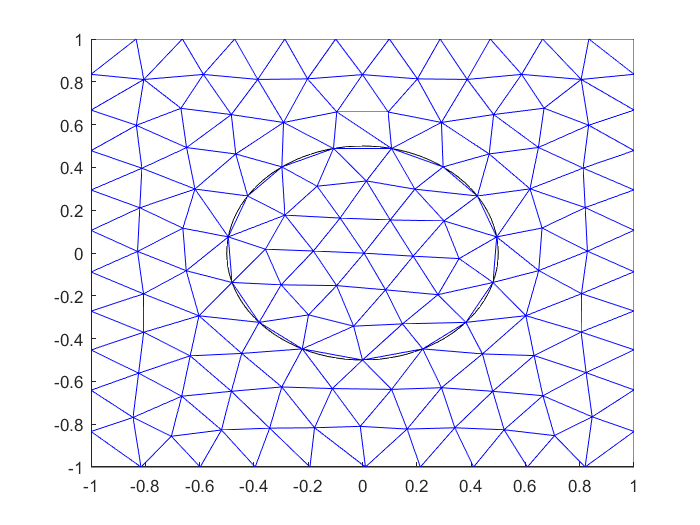}
	\end{minipage}
	\begin{minipage}{5cm}
		\includegraphics[width=5cm]{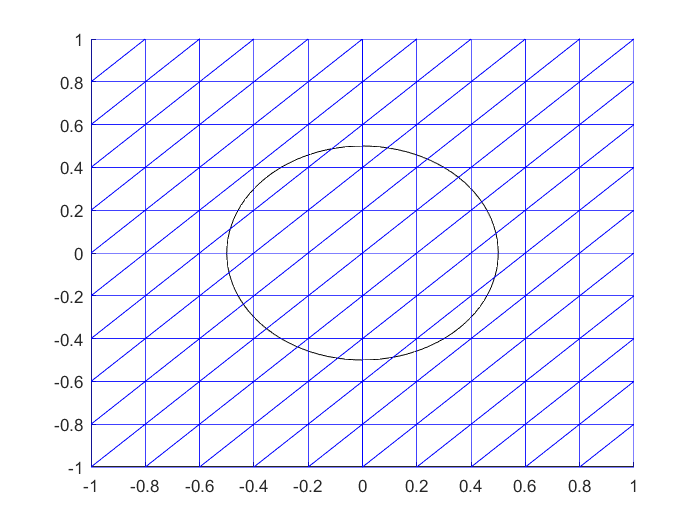}
	\end{minipage}
	\caption{Fitted(left) and Unfitted(right) mesh. The black line is the interface.}
\end{figure}
	
	For special case of the optimal control problem \eqref{eqobjective}-\eqref{eqconstraint}, which $\alpha$ is continuous in $\Omega$ and $g=0$. There are a lot of researches, see \cite{Casas86con,Falk73app,Tiba76err,Becker00ada} for early research, \cite{Li02ada,Hinze05var,Schneider16pos} for control constraint, \cite{Benedix09pos,Hintermuller10goa,Rosch12pri,Rosch12pos,Rosch17rel,Lius09new} for state constraint, \cite{Kohls14con,Gong17ada,Li02ada} for adaptive convergence analysis. However, there are few papers consider the optimal control problem governed by elliptic interface problems. To our best knowledge, IFEM $\cite{Zhang2015Immersed}$ is the first unfitted finite element method applied to this problem. But they only consider the homogeneous interface condition. If model has non-homogeneous interface condition, their method can not be directly extended. 	And in their numerical experiments, we also observe that it's convergence order of $H^1$ norm could reduce from 1 to 0.9, when the mesh refined. This phenomenon also mentioned in \cite{Lin2015Partially}, for some numerical examples, the convergence rates of classic IFEM in both the $H^1$ norm and the $L^2$ norm deteriorated, when the mesh becomes finer.

In this paper, we consider Nitsche-XFEM for a more general model with non-homogeneous interface condition. Nitshche-XFEM is simple, do not need to construct complicated shape functions to satisfy the interface condition compare with IFEM. And it only needs a little more degree of freedoms compare with standard linear finite element method. We followed variational discretization concept $\cite{Hinze05var,Hinze2009Variational}$ to discretize the continuous problem.
 Optimal error estimations are derived for the state, co-state and control in the mesh independent of the interface.  According to the discrete optimal conditions, we choose the proper algorithm to solve the optimal control problem. Numerical experiments confirmed our theoretical results, and also show that the convergence order of Nitsche-XFEM do not reduce, when we refine the mesh. 
	
	The remainder of this paper is organized as following. In section 2, we give some notations and the optimality conditions for the optimal control problem. In section 3, we give a brief introduction for Nitsche-XFEM, some theoretical results associate with Nitsche-XFEM are given. In section 4, we discretize optimal control problem, give the discrete optimality conditions and derive error estimates  for the state, co-state and control of the optimal control problem. In section 5, numerical experiments are given to confirm our theories. In section 6, some implementation details are provided. The last section is the conclusion.

	\section{Notation and optimality conditions}
	In this paper, we shall use the standard sobolev norm $\cite{Adams1975Sobolev}$ and $L^2$ inner product on $\Omega$, $\Omega_1$, $\Omega_2$ and $\Gamma$, we omit $\Omega$ when using the standard norm or $L^2$ inner product in $\Omega$. 
	And define some norms
$$\|u\|^2_{i,\Omega_1 \cup \Omega_2}:= \|u\|^2_{i,\Omega_1}+\|u\|^2_{i,\Omega_2} , \quad i=0,1,2.$$
	We assume $u$ is a function in $L^2(\Omega)$, and $g$ is a function in $H^{\frac12}(\Gamma)$.
Then weak formulation of state equation \eqref{eqstrongstate} is to find $y \in H^1_0(\Omega)$ such that
\begin{equation}
\label{eqweakform}
(\alpha \nabla y ,\nabla v)=(u,v)+(g,v)_{\Gamma}, \quad \forall v \in H^1_0(\Omega).
\end{equation} 
$(\alpha \nabla y ,\nabla v)$ denotes by $a(y,v)$. 

\begin{lemma}
	\label{leregularity}
	Problem \eqref{eqweakform} has a unique solution in $H^2(\Omega_1 \cup \Omega_2) \cap H^1_0(\Omega)$, and 
	$$\|y\|_1 +\|y\|_{2,\Omega_1 \cup \Omega_2} \leq C (\|u\|_{0}+\|g\|_{\frac12,\Gamma}).$$
	C is a positive constant independent of $u$ and $g$.
\end{lemma}

	Note that $C$ may change in this paper, but still a positive constant independent of the mesh size  $h$ and the location of interface $\Gamma$ relative to the mesh. $p \lesssim q$ denotes $p \leq C q$ and $p \gtrsim q$ denotes $p \geq C q$. $|x|$  denotes the area or length or absolute value of $x$.
	
	Let $u_a, \ u_b \in L^2(\Omega)$ we define $$U_{ad}:=\{ u\in L^2(\Omega):u_{a}\leq u \leq u_{b} \ \text{a.e. in} \ \Omega \}.$$ By using the standard technique in $\cite{Tr2010Optimal}$, we can easily derive the optimality condition of the optimal control problem \eqref{eqobjective}-\eqref{eqconstraint}.
	\begin{lemma}From \label{leoptimalcondition}
		\cite{Tr2010Optimal}.
		The optimal control problem of minimizing \eqref{eqobjective}-\eqref{eqconstraint} has a unique solution. $(y,p,u)\in H_0^1(\Omega)\times H_0^1(\Omega)\times U_{ad}$ such that 
\begin{align}
& a(y,v)=(u,v)_\Omega+(g,v),~~\forall v\in H_0^1(\Omega),\label{eqstate}\\
& a(v,p)=(y-y_d,v),~~\forall v\in H_0^1(\Omega),\label{eqcostate}\\
& (p+a u,v-u) \geq 0,~~\forall v\in U_{ad}.\label{eqprojection}
\end{align}
	\end{lemma}

	We call \eqref{eqcostate} as the co-state equation, and $p$ is the co-state or adjoint state. In addition, by use the Lemma $\ref{leregularity}$, we have $p\in H^2(\Omega_1 \cup \Omega_2)$. The variational inequality \eqref{eqprojection} is equivalent with a $L^2$ projection in $U_{ad}$, \text{i.e.}
	\begin{equation}\label{eqprojection2}
	u=P_{U_{ad}} (-\frac{1}{a} p).
	\end{equation} 
	
	\section{Nitsche-XFEM for state and co-state equations}
	\subsection{Extended finite element space}
	From the optimality condition, it is clear the state $y$ and co-state $p$ can be viewed as the solutions of two i!terface problems. Be aware of the solution's properties of interface problem, it is a natural idea to build a finite element space discontinuous across the interface $\Gamma$.
	We build the extended finite element space as following.
	
	Let $V_h^P$ be the standard linear finite element space with respect to triangulation $\mathcal{T}_h$, $\varphi_i$ is the nodal basis function of mesh points $P_i$, $i$ is the index of mesh points. We define the cut finite element space
	 $$V_h^\Gamma:=span\{\widetilde{ \varphi_{i} } :supp(\varphi_i) \cap \Gamma \ne \emptyset\}.$$ where the cut basis function
	$$\widetilde{ \varphi_{i} } = \left \{ \begin{array}{ll}
	0, & in \ \Omega_m,\\
	\varphi_{i}, & in  \ \Omega \setminus \Omega_m, 
	\end{array} \right. \ \text{if} \ P_i \ \text{in} \  \ \Omega_m. \quad m=1,2.$$
	It is clear that $\widetilde{ \varphi_{i} } \ne 0$ only at some elements, we call these elements as interface elements and other elements as non-interface elements. And every interface element has an intersection with the interface $\Gamma$.
	Then we define the extended finite element space as $$V_h:=\{v_h \in V_h^P \oplus V_h^\Gamma :v_h |_{\partial\Omega} =0 \}.$$ Notice the basis function is piecewise linear and continuous, so for any $v_h \in V_h$, $v_h$ is linear and continuous at $\Omega_1$ or $\Omega_2$.

	\begin{figure} [!hbtp]
		\label{fgspecialbase}
		\centering
		
		\begin{minipage}{3cm}		
			\includegraphics[width=3cm]{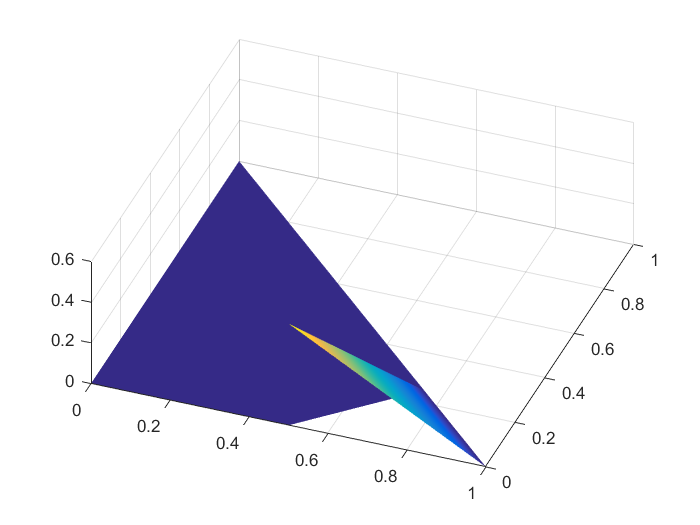}
		\end{minipage}
		\begin{minipage}{3cm}
			\includegraphics[width=3cm]{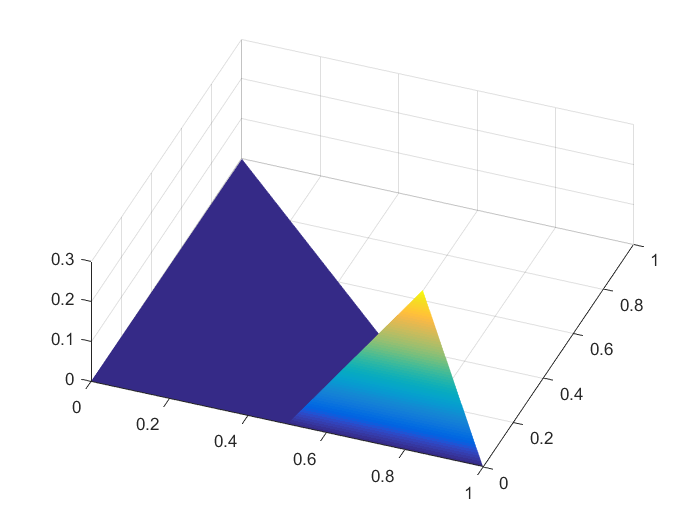}
		\end{minipage}
		\begin{minipage}{3cm}
			\includegraphics[width=3cm]{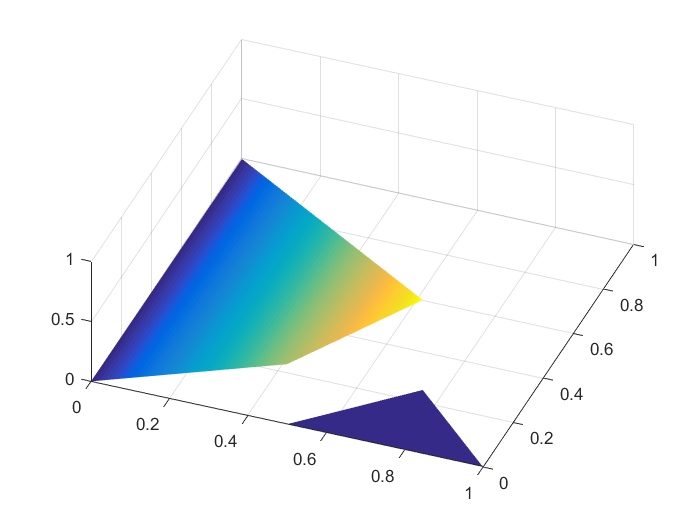}
		\end{minipage}
		\caption{	$\widetilde{\varphi_{i}}$ in 
			reference element $T$ with interface $y=x-\frac12 $. }
	\end{figure}
		
		\subsection{Scheme and results of Nitsche-XFEM}
		 
		In order to give the scheme and results of Nitsche-XFEM. We make the standard assumptions about the interface and conforming mesh like $\cite{Hansbo2002An}$.
		
		\noindent \textbf{A1:} The triangulation $\mathcal{T}_h$ is non-degenerate, i.e.
		$$\frac{h_T}{\rho_T}  \leq C, \quad \forall T \in \mathcal{T}_h.$$
		where $h_T$ is the diameter of $T$, $h= \mathop{max} \limits_{T \in \mathcal{T}_h} \{h_T\}$, and $\rho_T$ is the diameter of the largest ball contained in $T$. 
		
		\noindent \textbf{A2:} The interface $\Gamma$ intersects each element with two different points on different edges.
		
		\noindent \textbf{A3:} Let $\Gamma_T$ be the interface restrict in element $T$. And $\Gamma_{T,h}$ is the straight line segment connecting the points of intersection between interface $\Gamma$ and edges of $T$.  In local coordinates $(\xi,\eta)$,  $\Gamma_{T}$ and $\Gamma_{T,h}$ can be written in the following form,
		$$\Gamma_{T,h}=\{( \xi , \eta ) : 0 \le \xi \le | \Gamma_{T,h}|,  \eta=0 \},$$  
	 $$\Gamma_{T}=\{( \xi , \eta) : 0 \le \xi \le | \Gamma_{T,h}|,  \eta=\psi (\xi) \}.$$
		Where $\psi$ is a function.
		
		These assumptions are not strictly and  always fulfilled on the sufficiently fine mesh. To illustrate the method, we give some notations about Nitsche-XFEM. For each element $T$,  $$T_m:=T|_{\Omega_m}, \ k_m:=\frac {|T_m|}{|T|}, \ m=1,2.$$
	 and
	$$v_{m,h}:=v_h|_{\Omega_m}, \ \{\phi\}:=(k_1\phi_1+k_2\phi_2)|_{\Gamma}, \ \phi_m:=\phi|_{\Omega_m}.$$
	By means of a variant Nitsche's method in $\cite{Hansbo2002An}$, for state equation \eqref{eqstate} we have a symmetric positive scheme
	\begin{align} 
		( \alpha \nabla y^h,\nabla v_h)_{\Omega_1 \cup \Omega_2} -(\{ \alpha \nabla_\textbf{n} y^h\},[v_h])_{\Gamma}-(\{ \alpha \nabla_\textbf{n} v_h \},[y^h])_{\Gamma}+\lambda ([y^h],[v_h])_{\Gamma}= \nonumber \\
		(u,v_h)+( k_2 g,v_{1,h} )_{\Gamma}+( k_1 g,v_{2,h})_{\Gamma}, \quad \forall v_h \in V_h.  \label{eqnfem}
	\end{align}
	 
	Where $\lambda=\widetilde{C}h^{-1}\max\{\alpha_1,\alpha_2\}$ is the stable parameter, with positive constant $\widetilde{C}$ sufficiently large.
	\eqref{eqnfem} denotes by 
	\begin{equation} \label{eqaustate}
	a_h(y^h,v_h)=l(v_h). 
	\end{equation}
	We also have the discrete form of co-state equation \eqref{eqcostate}
	\begin{equation} \label{eqaucostate}
		a_h(v_h,p^h)=(y-y_d,v_h).
	\end{equation}
	 \begin{remark}
		In our view, terms $(\{ \alpha \nabla_\textbf{n} y_h\},[v_h])_{\Gamma}$ and $(\{ \alpha \nabla_\textbf{n} v_h \},[y_h])_{\Gamma}$ make the scheme \eqref{eqnfem} consistent and symmetric. The scheme is consistent in the sense $a_h(y-y_h,v_h)=0.$ And $\lambda ([y_h],[v_h])_{\Gamma}$ makes the scheme stable. Moreover, the method can be viewed as a combination of extended finite element space and Nitsche's method, that is why we call it as Nitsche-XFEM.
	\end{remark}
		Let us define the mesh dependent norm of Nitsche-XFEM $\cite{Hansbo2002An}$
	$$||| v ||| ^2:=\|\nabla v\|^2_{0,{\Omega_1 \cup \Omega_2}}+\|\{ \nabla_\textbf{n} v \}\|^2_{-1/2,\Gamma}+\|[v] \|^2_{1/2,\Gamma}, \\$$ 
	$$\|v\|^2_{1/2,\Gamma}:= \underset{T \in \mathcal{T}_h^{\Gamma}}  {\sum} h_T^{-1} \|v\|^2_{0,\Gamma_T},$$
	$$\|v\|^2_{-1/2,\Gamma}:= \underset{T \in \mathcal{T}_h^{\Gamma}}  {\sum} h_T \|v\|^2_{0,\Gamma_T}.$$
With assumptions A1-A3 and if we choose $\lambda$ large enough. We have the following lemmas and theorem. 
\begin{lemma} \label{lestablity}
	$ \cite{Hansbo2002An}$.
	The discrete form \eqref{eqnfem} is coercive on $V_h,$ i.e. 
	$$ a_h(v_h,v_h) \gtrsim ||| v_h |||^2 , \quad \forall v_h \in V_h, $$
	And, $$a_h(u_h,v_h)\leq ||| u_h ||| ,||| v_h |||,\quad \forall u_h,v_h \in V_h.$$
\end{lemma}

\begin{lemma} $ \cite{Hansbo2002An}$.
	\label{lepriorestimate}
	If $y$ is the solution of problem \eqref{eqweakform} and $y_h$ is the solution of problem \eqref{eqnfem}, then 
	$$||| y-y_h |||  \lesssim h\|y\|_{2,\Omega_1 \cup \Omega_2},$$ 
	$$\|y-y_h\|_0 \lesssim h^2\|y\|_{2,\Omega_1 \cup \Omega_2}.$$
\end{lemma}

\begin{theorem}
	\label{thdiscretepoincare}
	Discrete poincar\'{e} inequality, $ a_h(v_h,v_h)\gtrsim \|v_h\|_{0,\Omega_1 \cup \Omega_2}^2, \quad \forall v_h \in V_h$.
\end{theorem}
\begin{proof}
	We give a proof following from  $\cite{Chou2010Optimal}$. By $Lemma \ \ref{lestablity}$ , when $\lambda$ is large enough, we have
	$a_h(v_h,v_h)\gtrsim  |||v_h|||^2 \gtrsim  \| \nabla v_h\|^2_{0,{\Omega_1 \cup \Omega_2}}.$
	To prove the inequality, we want to choose a piecewise straight line path $I=\bigcup_i\{ x_i,x_{i+1}\}, \ i=0,1,...,n-1.$ $v_h$ is piecewise differentiable and continuous on the path. Clearly  $v_h$ is differentiable and continuous in any non-interface element, we can choose the path as straight line in these elements. In addition, by our assumption A1, we have
	\begin{equation}\label{eqs1}
		|T_p| \gtrsim |x_{i+1}-x_{i}|^2.
	\end{equation}
	$T_p$ is a non-interface element on the path. But $v_h$ is not differentiable and continuous across the interface. If path $I$ pass through the interface, we have to choose a point $x_*$ on the interface. Moreover, we can adjust the location of $x_*$ to make the $|x_{k+1}-x_*|$ and $|x_*-x_{k}|$ as short as possible, So the area of each part of interface element is bounded below by the $|x_*-x_{k}|^2$ and $|x_{k+1}-x_*|^2$ i.e. 
	\begin{equation}\label{eqs2}
		|T_q| \gtrsim  \max \left \{|x_*-x_{k}|^2,|x_{k+1}-x_*|^2 \right \}. 
	\end{equation}
 $T_q$ is any part of interface element.
	
	For any point $x \in \Omega$ there is a sequence of points $x_i$, and $x_0 \in \partial \Omega$. $v_h$ is piecewise continuous, linear and differentiable along the path $I=\bigcup_i\{ x_i,x_{i+1}\}.$  
	And if we choose the proper $x_*$ satisfy the above inequality. Because $v_h(x_0)=0$, by the mean value theorem, \eqref{eqs1}-\eqref{eqs2} and Cauchy-Schwartz inequality. We have
	
	\begin{align*}
	|v_h(x)|^2 &=\left |\sum^{n-1}_{i=0}(v_h(x_{i+1})-v_h(x_{i})) \right |^2 \\ 
	& =\left |\sum^{n-1}_{i=0} \nabla {v_h(\bar {x})(x_{i+1}-x_{i})} \right |^2 \\  
	& \leq  n \sum^{n-1}_{i=0} |\nabla {v_h(\bar {x}_i)}|^2h_i^2  \\ 
	& \lesssim  n\sum^{n-1}_{i=0} | \nabla v_h|^2_{0,T_i}. 
	\end{align*}
	
	Where $h_i=|x_{i+1}-x_{i}|$, $\bar{x}_i$ is a point in $[x_i,x_{i+1}].\\$
	Suppose $x$ is a point in $T$,  then $$ \int_T |v_h|^2 ds \lesssim nh^2\sum^{n-1}_{i=0} | \nabla v_h|^2_{0,T_i}.$$
	We have $nh\leq C$, since the number of line segments in the path is bounded by $Ch^{-1}$. So
	\begin{align}\nonumber
	nh^2 \sum^{n-1}_{i=0} | \nabla v_h|^2_{0,T_i} & \lesssim h\sum^{n-1}_{i=0} | \nabla v_h|^2_{0,T_i} \\ \nonumber
	 \int_T |v_h|^2 ds & \lesssim h \sum^{n-1}_{i=0} | \nabla v_h|^2_{0,T_i}. \nonumber
	\end{align}
	
	Summing the above inequality over $T$,  choose some paths that each $T_i$ appeared most $Ch^{-1}$ times at all the paths. Last, we obtain
	$$ \|v_h\|_{0,\Omega_1 \cup \Omega_2}^2 \lesssim \| \nabla v_h \|^2_{0,{\Omega_1 \cup \Omega_2}},$$
	which complete the proof.
\end{proof}

\section{Discrete optimal control problem}
\subsection{Discrete optimality conditions}

To discretize the optimal control problem, we followed the variational discretization concept in $\cite{Hinze05var}$. The idea is to discretize the state $y$ but not the control $u$ at first. The discretization of the optimal control problem \eqref{eqobjective}-\eqref{eqconstraint} is, find $y_h,u\in V_h \times U_{ad}$  minimizing
\begin{equation} \label{eqdobjective}
J_h(y_h,u)=\frac{1}{2} \int_{\Omega} (y_h-y_d)^2 ds+\frac{a}{2}\int_{\Omega} u^2 ds,
\end{equation} 
with
\begin{equation}\label{eqdstate}
a_h(y_h,v_h)=l(v_h), \quad \forall v_h \in V_h.
\end{equation}
Again, by the standard technique, we also have the discrete form of Lemma $\ref{leoptimalcondition}$. 
\begin{lemma}
	The optimal control problem of minimizing \eqref{eqdobjective}-\eqref{eqdstate} has a unique solution. $(y_h,p_h,u_h)\in V_h\times V_h\times U_{ad}$ such that 
\begin{align}
& a(y_h,v_h)=l_h(v_h), \quad \forall v_h\in V_h,\label{eqdstate2}\\
& a(v_h,p_h)=(y_h-y_d,v_h),\quad \forall v_h\in V_h,\label{eqdcostate}\\
& (p_h+a u_h,v-u_h) \geq 0,\quad \forall v\in U_{ad}.\label{eqdprojection}
\end{align}
\end{lemma}
Where $ l_h(v_h)=(u_h,v_h)+( k_2 g,v_{1,h} )_{\Gamma}+( k_1 g,v_{2,h})_{\Gamma}$. By variational equality \eqref{eqdprojection}, the control $u$ is discretized implicitly by the projection
\begin{equation}\label{eqprojection3}
u_h=P_{U_{ad}} (-\frac{1}{a} p_h).
\end{equation}
Moreover, if $u_a$ and $u_b$ are constant in $\Omega$ then \eqref{eqprojection3} equivalent to
\begin{equation}\label{eqprojection4}
u_h=min\{u_b,max\{u_a, -\frac{1}{a} p_h \}\}.
\end{equation}
Note $u_h$ may not in the extended finite element space $V_h$, but still in a finite dimension subspace of $U_{ad}$. So it is possible to solve the discrete system \eqref{eqdstate2}-\eqref{eqdprojection}. 

In order to get error estimates between the solutions $(y,p,u)$ and $(y_h,p_h,u_h)$, we recall the discrete form of state and co-state equation
\begin{align} 
&a_h(y^h,v_h)=l(v_h), \quad \forall v_h \in V_h, \label{eqastate} \\  
&a_h(v_h,p^h)=(y-y_d,v_h), \quad \forall v_h \in V_h. \label{eqacostate}
\end{align}
By Lemma $\ref{lepriorestimate}$ and regularity of $y$ and $p$, we have
$$|||y-y^h||| \lesssim h\|y\|_{2,\Omega_1 \cup \Omega_2}, \ \ \ \ \ \|y-y^h\|_0 \lesssim h^2\|y\|_{2,\Omega_1 \cup \Omega_2},$$ 
$$|||p-p^h||| \lesssim h\|p\|_{2,\Omega_1 \cup \Omega_2}, \ \ \ \ \ \|p-p^h\|_0 \lesssim h^2\|p\|_{2,\Omega_1 \cup \Omega_2}.$$  

\begin{theorem}\label{th1}
	Let $(y,p,u)\in H_0^1(\Omega)\times H_0^1(\Omega)\times U_{ad}$ and $(y_h,p_h,u_h)\in V_h\times V_h\times U_{ad}$ be the solutions to the continuous optimal control problem  \eqref{eqstate}-\eqref{eqprojection} and discrete optimal control problem \eqref{eqdstate2}-\eqref{eqdprojection}, respectively. And $(y^h,p^h)\in V_h\times V_h$ is the solutions of \eqref{eqastate}-\eqref{eqacostate}. Then we have 
	\begin{eqnarray}
	a^{\frac12} \|u-u_h\|_0+\|y-y_h\|_0 &\lesssim& \|y-y^h\|_0+{a^{-\frac12}}\|p-p^h\|_0, \label{t1} \\
	\|p-p_h\|_{0} &\lesssim& \|p-p^h\|_{0} + \|y-y_h\|_{0}, \label{t2} \\
	|||y-y_h||| &\lesssim& |||y-y^h|||+\|u-u_h\|_{0},  \label{t3} \\
	|||p-p_h|||&\lesssim& |||p-p^h|||+\|y-y_h\|_{0}.\label{t4}
	\end{eqnarray}
\end{theorem}
\begin{proof}
	First, by \eqref{eqdstate2}-\eqref{eqdcostate} and  \eqref{eqastate}-\eqref{eqacostate},
	we have
	\begin{align}
		a_h(y_h-y^h,v_h)&=(u_h-u,v_h), \quad \forall v_h \in V_h, \label{eqt1} \\
		a_h(v_h,p_h-p^h)&=(y_h-y,v_h), \quad \forall v_h \in V_h. \label{eqt2}
	\end{align}
	then
	\begin{align}
    (y_h-y,y_h-y^h)&=a_h(y_h-y^h,p_h-p^h) \label{eqt3} \\
 	&=(u_h-u,p_h-p^h)\label{eqt4}.
	\end{align}
 Set $v=u_h$ in \eqref{eqprojection} and $v=u$ in \eqref{eqdprojection}, we get
	\begin{align}
	(au+p,u_h-u) &\geq 0,  \label{eqva1}\\
 	(au_h+p_h,u-u_h) &\geq 0. \label{eqva2}
	\end{align}
Add together the above inequalities, we get $(a(u-u_h)+p-p_h,u_h-u) \geq 0.$
	Then by \eqref{eqt4},
	\begin{align*}
		 a\|u-u_h\|^2_0 & \leq (u_h-u,p-p_h) \\
		 &=(u_h-u,p-p^h)+(u_h-u,p^h-p_h) 	\\
		  &=(u_h-u,p-p^h)+(y_h-y,y_h-y^h)  \\
		  &\leq \frac{1}{2}(a\|u_h-u\|^2_0+\frac{1}{a}\|p-p^h\|^2_0)+(y_h-y,y_h-y^h) \\	
		  &\leq \frac{1}{2}(a\|u_h-u\|^2_0+\frac{1}{a}\|p-p^h\|^2_0) -\frac{1}{2}\|y-y_h\|_0^2+\frac{1}{2}\|y-y^h\|_0^2,
	\end{align*}
	which implies \eqref{t1}. 
	
	Second let's show \eqref{t2}, from Lemma $\ref{lestablity}$, Theorem $\ref{thdiscretepoincare}$ and \eqref{eqt2}, we have
\begin{align*}
	|| p_h-p^h||^2_{0}&\lesssim |||p_h-p^h|||^2\\
	&\lesssim a_h(p_h-p^h,p_h-p^h)=(y_h-y, p_h-p^h)\\
	&\lesssim ||y_h-y||_{0}|| p_h-p^h||_{0}.
\end{align*}
By triangle inequality
\begin{align*}
\|p-p_h\|_0 & \leq \|p-p^h\|_0+\|p^h-p_h\|_0 \\
& \lesssim \|p-p^h\|_0+\|y_h-y\|_0 .
\end{align*}
The last equality is \eqref{t2}.

Third, let us show \eqref{t3}. From Lemma $\ref{lestablity}$, Theorem $\ref{thdiscretepoincare}$ and \eqref{eqt1}, we obtain
\begin{align*}
|||y_h-y^h|||^2 &\lesssim a_h(y_h-y^h,y_h-y^h)= (u_h-u,y_h-y^h) \\
&\leq \|u-u_h\|_{0} \|y-y^h\|_0 \\
&\lesssim \|u-u_h\|_{0}|||y_h-y^h|||.
\end{align*}
Together with triangle inequality, we have \eqref{t3}.

Finally, let us show \eqref{t4}. From Lemma \ref{lestablity}, Theorem $\ref{thdiscretepoincare}$ and \eqref{eqt2}, we get
\begin{align*}
	|||p_h-p^h|||^2&\lesssim a_h(p_h-p^h,p_h-p^h)=(y_h-y, p_h-p^h)\\
	&\lesssim||y_h-y||_{0}|| p_h-p^h||_{0}\\
	&\lesssim ||y_h-y||_{0}||| p_h-p^h|||,
\end{align*}
which, together with the triangle inequality, indicates  \eqref{t4}.
\end{proof}

 We immediately have the following error estimates for the optimal control problem.
\begin{theorem}\label{th2}
	Let $(y,p,u)\in \left(H_0^1(\Omega)\cap H^2 (\Omega_1 \cup \Omega_2)\right)\times \left(H_0^1(\Omega)\cap H^2 (\Omega_1 \cup \Omega_2)\right)\times U_{ad}$ and $(y_h,p_h,u_h)\in V_h\times V_h\times U_{ad}$ be  the solutions to the continuous problem (\ref{eqstate})-(\ref{eqprojection}) and the discrete problem (\ref{eqdstate2})-(\ref{eqdprojection}), respectively. Then we have 
	\begin{eqnarray}
	\|u-u_h\|_{0}+\|y-y_h\|_{0} +\|p-p_h\|_{0} &\lesssim& h^2 (\|y\|_{2,\Omega_1\cup \Omega_2}+\|p\|_{2,\Omega_1\cup \Omega_2}),\label{u0}\\
	|||y-y_h||| + |||p-p_h|||&\lesssim& h (\|y\|_{2,\Omega_1\cup \Omega_2}+\|p\|_{2,\Omega_1\cup \Omega_2}). \label{y00}
	\end{eqnarray}
	Moreover, if $u$ is unconstrained, $i.e.$ $U_{ad}=L^2(\Omega)$. We also have $$|||u-u_h|||\lesssim h (\|y\|_{2,\Omega_1\cup \Omega_2}+\|p\|_{2,\Omega_1\cup \Omega_2}). \label{u00}$$
\end{theorem}

	\section{Numerical results}
	In this section, we consider the optimal control problem with the following state equation. 
	\begin{equation}\label{eqstrongstate2}
	\left\{
	\begin{array}{rll}
	& -\nabla\cdot(\alpha(x)\nabla y)=u+f, & \text{ in }\Omega \\
	& y=y_b, & \text{ on }\partial\Omega \\
	& [y]=0,[\alpha\nabla_n y ]=g, & \text{ on }\Omega\\
	\end{array}
	\right.
	\end{equation}
Where $f\in L^2(\Omega)$ and $y_b\in H^{\frac{3}{2}}(\partial\Omega)$. In this case, it is easy to check our theory still valid.

We construct the optimal control problem with analytical state, co-state and control like $\cite{Tr2010Optimal}$. The procedure is,  give the analytical $y$, $p$ satisfy the interface and boundary conditions at first, second we compute the corresponding $u$, $y_d$ and $f$ by
\begin{align*} 
	&u=P_{U_{ad}}(-\frac{1}{a}p),\\
	&- \nabla \cdot (\alpha(x)\nabla y)=u+f, \\
	&- \nabla \cdot (\alpha(x)\nabla p)=y-y_d. 
\end{align*}
Then we use the method in this paper to discretize the optimal control problem and choose proper algorithms to solve the discrete system. See next section for more details about implementation.

In all numerical examples we use N $\times$ N uniform triangular mesh, and show the errors in $H^1$ semi-norm and $L^2$ norm. 

		\noindent \textbf{Example 1. Segment interface} 
\quad Interface $\Gamma$ is a line segment,
$$x_2=kx_1+b, \ k= -\frac{\sqrt{3}}{3},\ b=\frac{(6+\sqrt{6}-2\sqrt{3})}{6}.$$
 $\Omega$ is a domain $[0,1]\times[0,1]$,
  $$\Omega_1=\{(x_1,x_2):kx_1+b-x_2>0 \} \cap \Omega , \quad \Omega_2=\{(x_1,x_2):kx_1+b-x_2\le0 \} \cap \Omega.$$
   We choose $\alpha_1=1$, $\alpha_2=100$, stable parameter $\lambda=1000h^{-1}$, regulation parameter $a=0.01$ and $U_{ad}=L^2(\Omega)$.
	\begin{eqnarray}
	u(x_1,x_2)=
	\begin{cases}
	(x_2-kx_1-b)x_1(x_1-1)x_2(x_2-1)sin(x_1x_2),& \mbox{ in $ \Omega_1$ } \\
	100(x_2-kx_1-b)x_1(x_1-1)x_2(x_2-1)sin(x_1x_2),&\mbox{ in $ \Omega_2$ }
	\end{cases}
	\end{eqnarray}
	\begin{eqnarray}
	y(x_1,x_2)=
	\begin{cases}
	\frac{(x_2-kx_1-b)cos(x_1x_2)}{200}+(x_2-kx_1-b)^3, & \mbox{ in $ \Omega_1$ } \\
	\frac{(x_2-kx_1-b)cos(x_1x_2)}{2}, &\mbox{ in $ \Omega_2$ }
	\end{cases}
	\end{eqnarray}
	\begin{eqnarray}
	p(x_1,x_2)=-au(x_1,x_2).
	\end{eqnarray}
	
	\begin{table}[!hbtp]
		\centering
		\begin{tabular}{|c|c|c|c|c|c|c|}
			\hline
			N & $\frac{\|u-u_h\|_0}{\|u\|_0}$ & order & $\frac{\|y-y_h\|_0}{\|y\|_0}$ & order & $\frac{\|p-p_h\|_0}{\|p\|_0}$& order  \\
			\hline
			16 &3.9941e-02&      &8.7667e-03&      &3.9941e-02&      \\
			\hline
			32 &9.6399e-03&2.1&2.1955e-04&2.0&9.6399e-03&2.1\\
			\hline
			64 &2.3780e-03&2.0&5.5005e-04&2.0&2.3780e-03&2.0\\
			\hline
			128&5.9194e-04&2.0&1.3834e-05&2.0&5.9195e-04&2.0\\
			\hline
			256&1.4794e-04&2.0&3.5203e-06&2.0&1.4794e-04&2.0\\
			\hline
		\end{tabular}
		\caption{$L^2$ errors for Example 1.}
	\end{table}
	\begin{table}[!hbtp]
		\centering
		\begin{tabular}{|c|c|c|c|c|c|c|}
			\hline
			N & $\frac{|u-u_h|_1}{|u|_1}$ & order & $\frac{|y-y_h|_1}{|y|_1}$ & order & $\frac{|p-p_h|_1}{|p|_1}$& order  \\
			\hline
			16 &2.0695e-01&       &1.0180e-01&       &2.0695e-01&      \\
			\hline
			32 &1.0404e-01&1.0&5.0958e-02&1.0&1.0404e-01&1.0\\
			\hline
			64 &5.2064e-02&1.0&2.5486e-02&1.0&5.2064e-02&1.0\\
			\hline
			128&2.6035e-02&1.0&1.2744e-02&1.0&2.6035e-02&1.0\\
			\hline
			256&1.3017e-02&1.0 &6.3722e-03&1.0&1.3017e-02&1.0 \\
			\hline
		\end{tabular}
		\caption{$H^1$ errors for Example 1.}
	\end{table}
	
%

	\begin{remark}
		 Although the straight line interface is not strictly in $\Omega$, we choose the analytical solution both in $H^2(\Omega_1 \cup \Omega_2)$, which is in line with our theory.
	\end{remark}
	
	\noindent \textbf{Example 2. Circle Interface} 
	\quad Interface $\Gamma$ is a circle, centered at $(0,0)$ with radius $r= \frac{\sqrt{3}}{4}$. $\Omega$ is a domain $[-1,1]\times[-1,1]$. 
	$$\Omega_1=\{(x_1,x_2):x_1^2+x_2^2\le r^2 \}, \quad \Omega_2=\{(x_1,x_2):x_1^2+x_2^2>r^2 \}\cap \Omega.$$ We choose stable parameter $\lambda=5000h^{-1}$, $\alpha_1=1,$ $\alpha_2=1000$. We choose regulation parameter $a=1$, and $U_{ad}=\{u\in L^2(\Omega):-\frac12 \leq u \leq \frac12 \ \text{ a.e in} \ \Omega \}$.

	 \begin{eqnarray}
	 \varphi(x_1,x_2)=
	 \begin{cases}
	 \frac{5(x_1^2+x_2^2-r^2)(x_1^2-1)(x_2^2-1)}{\alpha_1}, & \mbox{ in $ \Omega_1$ } \\
	 \frac{5(x_1^2+x_2^2-r^2)(x_1^2-1)(x_2^2-1)}{\alpha_2},  &\mbox{ in $ \Omega_2$ }
	 \end{cases}
	 \end{eqnarray}
	 
	 \begin{eqnarray}
	 u(x_1,x_2)= \min \left \{\frac12,\max \left \{-\frac12,\varphi (x_1,x_2) \right \}  \right \},
	 \end{eqnarray}
	 
	 \begin{eqnarray}
	 y(x_1,x_2)=
	 \begin{cases}
	 \frac{(x_1^2+x_2^2)^{\frac32}}{\alpha_1}-10(x_1^2+x_2^2-r^2)sin(x_1x_2), & \mbox{ in $ \Omega_1$ } \\
	 \frac{(x_1^2+x_2^2)^{\frac32}}{\alpha_2} +( \frac{1}{\alpha_1}-\frac{1}{\alpha_2})r^3, &\mbox{ in $ \Omega_2$ }
	 \end{cases}
	 \end{eqnarray}
	 
	 \begin{eqnarray}
	 p(x_1,x_2)=-a\varphi(x_1,x_2).
	 \end{eqnarray}
	 
	 \begin{table}[!hbtp]
	 	\centering
	 	\begin{tabular}{|c|c|c|c|c|c|c|}
	 		\hline
	 		N & $\frac{\|u-u_h\|_0}{\|u\|_0}$ & order & $\frac{\|y-y_h\|_0}{\|y\|_0}$ & order & $\frac{\|p-p_h\|_0}{\|p\|_0}$& order  \\
	 		\hline
	 		16 &4.4640e-02&       &6.7792e-02&      &5.9076e-02&      \\
	 		\hline
	 		32 &1.7953e-02&1.3&2.3134e-02&1.6&1.8254e-02&1.7\\
	 		\hline
	 		64 &3.9458e-03&2.2&5.7710e-03&2.0&3.9865e-03&2.2\\
	 		\hline
	 		128&7.7806e-04&2.3&1.3023e-03&2.2&8.2130e-04&2.3\\
	 		\hline
	 		256&1.2751e-04&2.6&2.0961e-04&2.6&1.5615e-04&2.4\\
	 		\hline
	 	\end{tabular}
	 	\caption{$L^2$ errors for Example 2.}
	 \end{table}
	 
	 \begin{table}[!hbtp]
	 	\centering
	 	\begin{tabular}{|c|c|c|c|c|}
	 		\hline
	 		N & $\frac{|y-y_h|_1}{|y|_1}$ & order & $\frac{|p-p_h|_1}{|p|_1}$& order  \\
	 		\hline
	 		16 &5.0048e-01&       &2.0831e-01&\\
	 		\hline
	 		32 &2.4468e-01&1.0&1.0421e-01&1.0\\
	 		\hline
	 		64 &1.1515e-01&1.1&4.9146e-02&1.1\\
	 		\hline
	 		128&5.7116e-02&1.0&2.4365e-02&1.0\\
	 		\hline
	 		256&2.6058e-02&1.1&1.1514e-02&1.1\\
	 		\hline
	 	\end{tabular}
	 	\caption{$H^1$ errors for Example 2.}
	 \end{table}
 
 We show the Nitsche-XFEM solutions of control, state, co-state and the boundary of active set in 32 $\times$ 32 uniform triangular mesh. See Fig. \ref{figcontrol}-\ref{figactiveset}.
 \begin{figure}[!hbtp]
 	\centering
 	
 	\begin{minipage}{5cm}
 		\includegraphics[width=5cm]{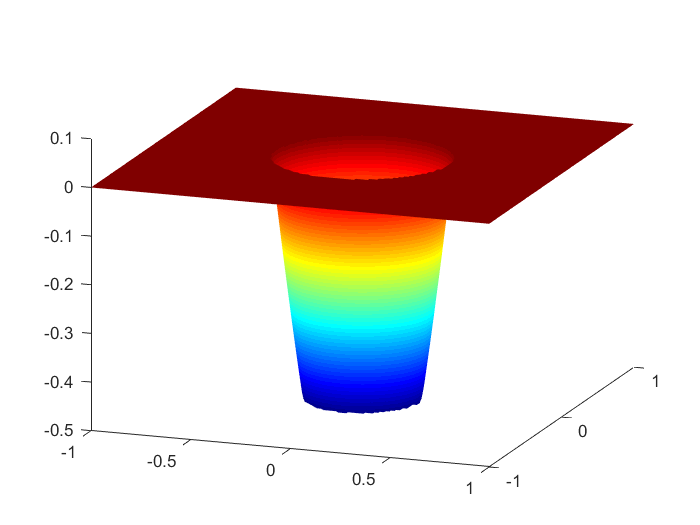}
 	\end{minipage}
 	\begin{minipage}{5cm}
 		\includegraphics[width=5cm]{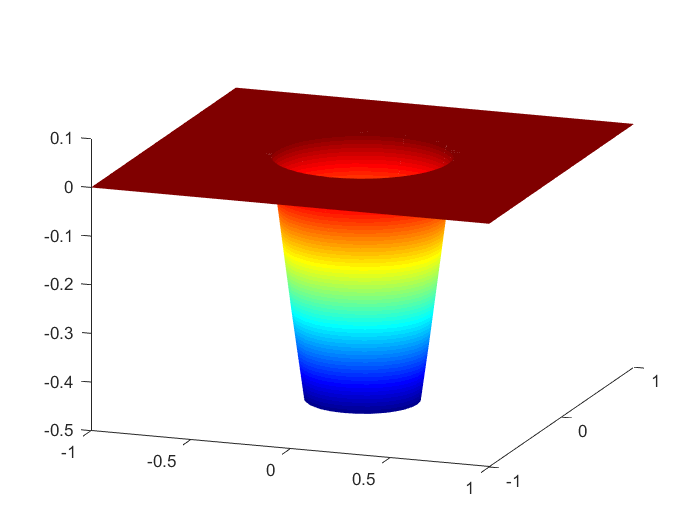}
 	\end{minipage}
 	\caption{ The exact(left) and Nitshce-XFEM(right) solutions of control for Example 2.}
 	\label{figcontrol}
 \end{figure}

 \begin{figure}[!hbtp]
	\centering
	
	\begin{minipage}{5cm}
		\includegraphics[width=5cm]{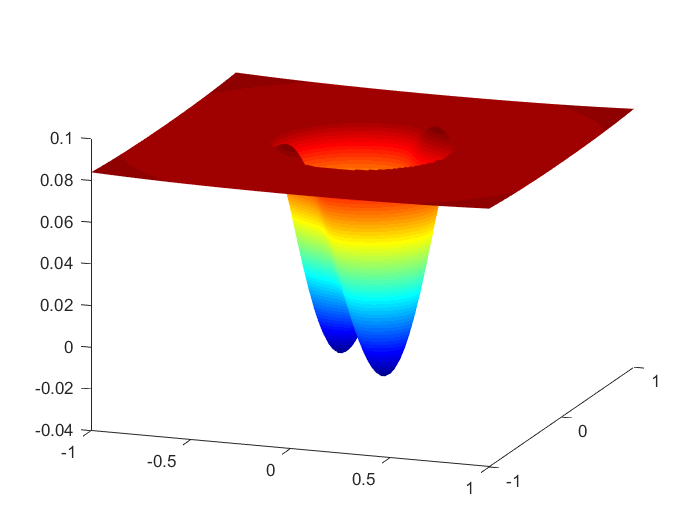}
	\end{minipage}
	\begin{minipage}{5cm}
		\includegraphics[width=5cm]{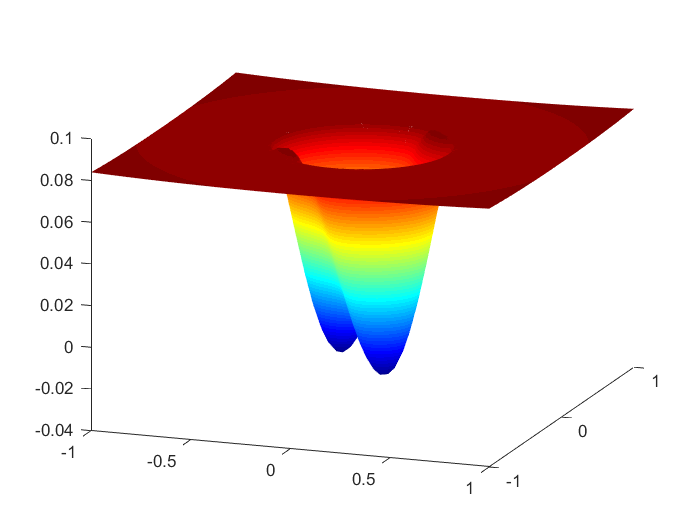}
	\end{minipage}
	\caption{ The exact(left) and Nitshce-XFEM(right) solutions of state for Example 2.}
	\label{figstate}
\end{figure}

 \begin{figure}[!hbtp]
	\centering
	\includegraphics[width=10cm]{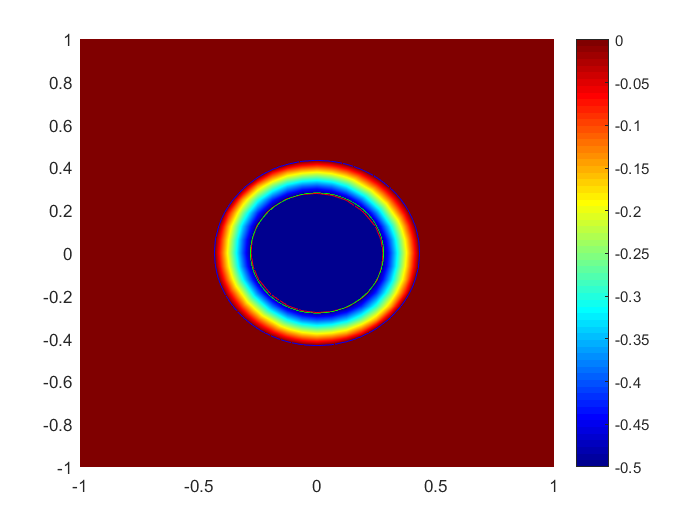}
	\caption{The Nitsche-XFEM solution of control for Example 2, green and red line are boundaries of exact and computed active set respectively, blue line is the interface $\Gamma_h$.}
		\label{figactiveset}
\end{figure}
 \begin{remark}
 	Notice that we use the piecewise straight line $\Gamma_h$  as the interface instead of $\Gamma$ in this numerical example. The approximation is more and more accurate when refine the mesh.
 \end{remark}

\noindent \textbf{Example 3. Compare with IFEM}
	\quad We also compare Nistche-XFEM with IFEM for the optimal control problem with interface. Consider the following example,
	interface $\Gamma$ is a circle, centered at $(0,0)$ with radius $r= \frac12$. $\Omega$ is a domain $[-1,1]\times[-1,1]$.
	 $$\Omega_1=\{(x_1,x_2):x_1^2+x_2^2\le r^2 \}, \quad \Omega_2=\{(x_1,x_2):x_1^2+x_2^2>r^2 \}\cap \Omega.$$ We choose stable parameter $\lambda=10000h^{-1}$, $\alpha_1=1,$ $\alpha_2=10$, regulation parameter $a=0.01$, and $U_{ad}=L^2(\Omega)$.
	
	\begin{eqnarray}
	u(x_1,x_2)=
	\begin{cases}
	\frac{5(x_1^2+x_2^2-r^2)(x_1^2-1)(x_2^2-1)}{\alpha_1},& \mbox{ in $ \Omega_1$ } \\
	\frac{5(x_1^2+x_2^2-r^2)(x_1^2-1)(x_2^2-1)}{\alpha_2},  &\mbox{ in $ \Omega_2$ }
	\end{cases}
	\end{eqnarray}

	\begin{eqnarray}
	y(x_1,x_2)=
	\begin{cases}
	\frac{(x_1^2+x_2^2)^{\frac32}}{\alpha_1}, & \mbox{ in $ \Omega_1$ } \\
	\frac{(x_1^2+x_2^2)^{\frac32}}{\alpha_2} +( \frac{1}{\alpha_1}-\frac{1}{\alpha_2})r^3, &\mbox{ in $ \Omega_2$ }
	\end{cases}
	\end{eqnarray}
	
	\begin{eqnarray}
	p(x_1,x_2)=-au(x_1,x_2).
	\end{eqnarray}
	
\begin{table}[!hbtp]
	\centering
	\begin{tabular}{|c|c|c|c|c|c|c|}
		\hline
		N & $\|u-u_h\|_0$ & order & $\|y-y_h\|_0$ & order & $\|p-p_h\|_0$& order  \\
		\hline
		16 &1.1316e-02&       &4.4535e-03&      &1.1316e-04&      \\
		\hline
		32 &3.0688e-03&1.88&1.1883e-03&1.91&3.0688e-05&1.88\\
		\hline
		64 &7.5979e-04&2.01&3.1686e-04&1.91&7.5979e-06&2.01\\
		\hline
		128&1.8516e-04&2.04&7.6393e-05&2.05&1.8516e-06&2.04\\
		\hline
		256&4.2966e-05&2.11&1.8584e-05&2.04&4.2966e-07&2.11\\
		\hline
	\end{tabular}
	\caption{$L^2$ errors of Nitsche-XFEM for Example 3.}
\end{table}

\begin{table}[!hbtp]
	\centering
	\begin{tabular}{|c|c|c|c|c|c|c|}
		\hline
		N & $|u-u_h|_1$ & order & $|y-y_h|_1$ & order & $|p-p_h|_1$& order  \\
		\hline
		16 &1.1407e-01&       &1.1311e-01&      &1.1401e-03&      \\
		\hline
		32 &5.7015e-02&1.00&5.8796e-02&0.94&5.6926e-04&1.00\\
		\hline
		64 &2.7869e-02&1.03&2.9448e-02&1.00&2.7932e-04&1.03\\
		\hline
		128&1.3830e-02&1.01&1.4800e-02&0.99&1.3852e-04&1.01\\
		\hline
		256&6.8465e-03&1.01&7.3659e-03&1.00&6.8465e-05&1.01\\
		\hline
	\end{tabular}
	\caption{$H^1$ errors of Nitsche-XFEM for Example 3.}
\end{table}

\begin{table}[!hbtp]
	\centering
	\begin{tabular}{|c|c|c|c|c|c|c|}
		\hline
		N & $\|u-u_h\|_0$ & order & $\|y-y_h\|_0$ & order & $\|p-p_h\|_0$& order  \\
		\hline
		16 &1.1889e-02&       &4.6400e-03&      &1.1889e-04&      \\
		\hline
		32 &3.1406e-03&1.92&1.2288e-03&1.91&3.1406e-05&1.92\\
		\hline
		64 &7.0663e-04&2.15&3.1438e-04&1.96&7.0663e-06&2.15\\
		\hline
		128&1.6334e-04&2.11&8.1934e-05&1.93&1.6334e-06&2.11\\
		\hline
		256&3.5894e-05&2.18&2.1650e-05&1.92&3.5894e-07&2.18\\
		\hline
	\end{tabular}
	\caption{$L^2$ errors of IFEM for Example 3. (From $\cite{Zhang2015Immersed}$)}
\end{table}

\begin{table}[!hbtp]
	\centering
	\begin{tabular}{|c|c|c|c|c|c|c|}
		\hline
		N & $|u-u_h|_1$ & order & $|y-y_h|_1$ & order & $|p-p_h|_1$& order  \\
		\hline
		16 &1.0665e-01&       &1.0778e-01&      &1.0665e-03&      \\
		\hline
		32 &5.2602e-02&1.01&5.5660e-02&0.95&5.2602e-04&1.01\\
		\hline
		64 &2.7054e-02&0.95&2.9084e-02&0.93&2.7054e-04&0.95\\
		\hline
		128&1.4028e-02&0.94&1.5047e-02&0.95&1.4028e-04&0.94\\
		\hline
		256&7.4170e-03&0.91&7.9081e-03&0.92&7.4170e-05&0.91\\
		\hline
	\end{tabular}
	\caption{$H^1$ errors of IFEM for Example 3. (From $\cite{Zhang2015Immersed}$)}
\end{table}
In this example, the convergence order of Nitsche-XFEM is always full when the mesh is fine enough, our other numerical examples also show the fact. While for IFEM, the convergence order of $H^1$ semi-norm reduced from 1.01 to 0.91 when the mesh is refined. 

	\section{Implementation aspects}
	
	In this section, we provide some details in our numerical experiments.
	If there is no constraint on the control i.e. $U_{ad}=L^2(\Omega)$. We solve the optimal control problem with state equation \eqref{eqstrongstate2}, by solve the following linear system.
	\begin{align}
	&a_h(y_h,v_h)=l_h(v_h)+(f,v_h),\quad \forall v_h \in V_h. \label{eqi1} \\
		&a_h(p_h,v_h)=(y_h-y_d,v_h), \quad \forall v_h \in V_h. \label{eqi2}\\
		&p_h=-a u_h, \label{eqi3} \\
		&y_h=b_h, \ p_h=0, \quad on \ \partial \Omega. \label{eqi4} 
	\end{align}
Where $b_h$ is the discrete boundary condition. 

In order to illustrate how to solve the linear system, let us introduce some notations about the matrix and vector.  The basis function of Nitsche-XFEM denote by $\phi_i$. 
\begin{align*}
	&A(i,j)=a_h(\phi_i,\phi_j),\\
	&M(i,j)=(\phi_i,\phi_j),\\
	&F_1(j)=(k_2 g,\phi_{j,1})_{\Gamma_h}+(k_1 g,\phi_{j,2})_{\Gamma_h}+(f,\phi_j),\\
	&F_2(j)=-(y_d,\phi_j) .
\end{align*}
Where  $i,j=1,2,3...,S$,  $S$ is the number of total degree of freedoms, $A$ is the stiffness matrix, $M$ is the mass matrix, $F_1$ and $F_2$ are column vectors.
$(U_h,Y_h,P_h)$ are column vectors consisting of corresponding degree of freedoms $(u_h,y_h,p_h)$.
With these notations, the linear system \eqref{eqi1}-\eqref{eqi4} rewritten to,
\begin{align*}
	AY_h&=MU_h+F_1 ,\\
	AP_h&=MY_h+F_2, \\
	P_h&=-aU_h .
\end{align*}
In matrix form,
\begin{center}
$\begin{bmatrix}
-aM & \textbf{0} &-M\\
\textbf{0} & -M & A \\
-M & A& 	\textbf{0}
\end{bmatrix}$
$
\begin{bmatrix}
U_h \\
Y_h \\
P_h
\end{bmatrix}
$
$ =
\begin{bmatrix}
\textbf{0} \\
F_2 \\
F_1
\end{bmatrix}$ .
\end{center}

 The last is to apply boundary condition \eqref{eqi4}
to this system and solve it. Note we have to use special numerical integration scheme to get the stiffness matrix and mass matrix, because $\alpha$ and some basis functions are discontinuous across the interface. We simple divided the interface element into several sub-triangles and integrate these discontinuous functions on sub-triangles, see Fig. \ref{fgin1}.
\begin{figure}[htbp] 

	\centering
	\includegraphics[width=10cm]{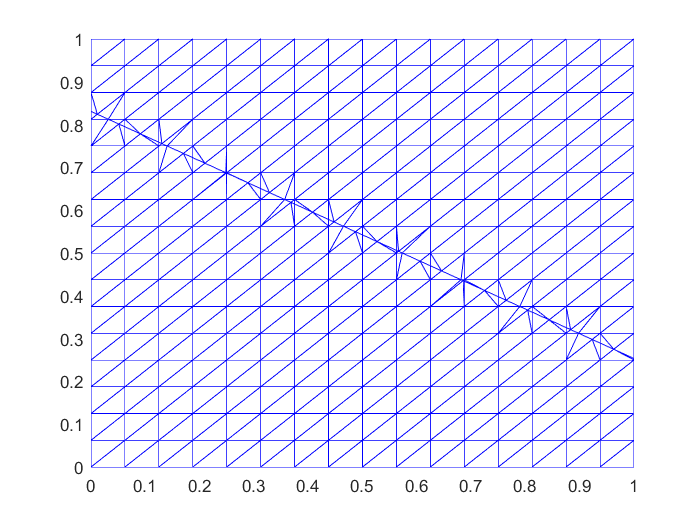}
		\caption{Numerical integration mesh for Example 1}
			\label{fgin1}
\end{figure}

If $u$ is constrained, for example $U_{ad}=\{u\in L^2(\Omega): -\frac12 \leq u \leq \frac12 \ \text{a.e. in} \  \Omega \}$. Then we can not use the above algorithm to solve the optimal control problem. The system became to a non-linear and non-smooth system.
\begin{align}
&a_h(y_h,v_h)=l_h(v_h)+(f,v_h), \quad \forall v_h \in V_h,\label{eqii1}\\
&a_h(p_h,v_h)=(y_h-y_d,v_h), \quad \forall v_h \in V_h,\label{eqii2}\\
&u_h=\min \left \{\frac12,\max \left \{-\frac12,-\frac{p_h}{a} \right \} \right \}, \label{eqii3}\\
&y_h=b_h,\ \ p_h=0,\quad \text{on} \ \partial \Omega. \label{eqii4}
\end{align}

Because for any $p_h \in V_h$, $u_h=\min \left \{\frac12,\max \left \{-\frac12,-\frac{p_h}{a} \right \} \right \}$ may not in $V_h$. It makes the system difficult (or expensive) to solve in some cases, for example it is expensive when use high order finite element method to solve the system $\cite{Sevilla2010Polynomial}$. While the solution of Nitsche-XFEM is piecewise linear, we can easily apply fixed-point iteration or semi-smooth newton method  $\cite{Hinze2012A}$ to solve the above system \eqref{eqii1}-\eqref{eqii4}. We give a description of fixed-point iteration.
\paragraph{Algorithm}
\begin{enumerate}
	\item Initialize $u^i_h=u^0$; 
	\item Compute $y^i_h\in V_h$ by $a_h(y_h^i,w_h)=l^i_h(v_h),\forall v_h\in V_h$; \label{computeyhi}
	\item Compute $p_h^i\in V_h$ by $a_h(v_h,p_h^i)=(y_h^i-y_d,v_h),\forall v_h\in V_h$;
	\item Set $u_h^{i+1}=\max\{u_a,\min\{-\frac{p_h}{a},u_b\}\}$;
	\item if $|u_h^{i+1}-u_h^i|<\text{Tol}$ or $i+1>\text{MaxIte}$, then output $u_h=u_h^{i+1}$, else $i=i+1$, and go back to Step \ref{computeyhi}.
\end{enumerate}
Where $u^0$ is an initial value, Tol is the tolerance, MaxIte is the maximal iteration number and $l^i_h(v_h)=(u_h^i,v_h)+(f,v_h)+(k_2g,v_{1,h})_{\Gamma_h}+(k_1g,v_{2,h})_{\Gamma_h}$. This algorithm is convergent when the regularity parameter a is large enough (cf. \cite{M2008}). If the regularity parameter $a$ is small, we recommend semi-smooth newton method.

In step \ref{computeyhi}, we also need special numerical integration scheme when computing $(u_h^i,v_h)$. Because $u_h^i 
\in U_{ad}$ has constraints, see Fig. \ref{meshcircle}. It should be noticed that the numerical integration mesh is used only for the numerical integration.

	\begin{figure}[htbp]

	\centering
		\includegraphics[width=10cm]{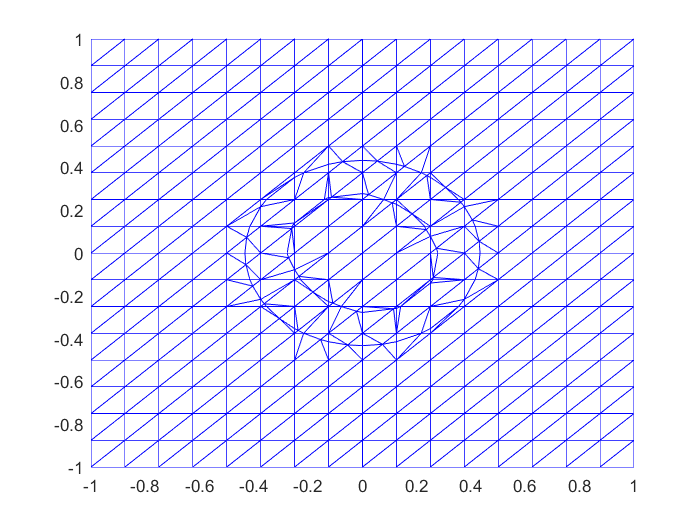}
	\caption{Numerical integration mesh for Example 2.}
		\label{meshcircle}
\end{figure}

\section{Conclusion}
In this paper, a numerical method has developed for optimal control problem govern by elliptic PDEs with interfaces. Optimal error estimations are derived for the state, co-state and control in an unfitted mesh.  Our method is suitable for the model with non-homogeneous interface condition. Numerical results show our method is more stable for some problems compare with IFEM, stable means the convergence order do not reduce, when refine the mesh. There are many different XFEMs for different problems, but they are not used to solve the optimal control problem. We want to consider other XFEM for other optimal control problem, for example, XFEM for optimal control problem govern by elliptic PDEs in a non-convex domain.


\end{document}